\documentclass[a4paper,11pt, twoside]{article}

\usepackage{latexsym}
\usepackage{amsmath}
\usepackage{ntheorem}
\usepackage{amssymb}
\usepackage{abstract}
\usepackage{titletoc}
\usepackage{enumerate}
\usepackage{graphicx}
\usepackage{subfigure}
\usepackage[english]{babel}
\usepackage{times}
\usepackage[T1]{fontenc}
\usepackage{theorem}
\usepackage{multicol}
\usepackage{graphicx}
\contentsmargin{1pt}

\dottedcontents{section}[2.3em]{}{2.3em}{5pt}
\dottedcontents{subsection}[5.5em]{}{3.2em}{5pt}

\newtheorem{theorem}{Theorem}[section] 
\newtheorem{definition}[theorem]{Definition} 
\newtheorem{lemma}[theorem]{Lemma} 
\newtheorem{corol}[theorem]{Corollary}

\newcommand{\card}{\rm{card}}
\newcommand{\Poo}{\mathrm{P}}
\newcommand{\Eoo}{\mathrm{E}}

\newenvironment{proof}{{{\noindent\it{Proof.}}\ \ }}{\hfill $\square$}

\usepackage{color}

\title{A tail estimate for empirical processes of multivariate Gaussian under general dependence} 
\author{Wen Huo\footnote{E-mail: huowen120513@fuji.waseda.jp} \ and Yasutaka Shimizu\footnote{E-mail: shimizu@waseda.jp} \\
{\it Department of Applied Mathematics, Waseda University}}
\date{March 20, 2023}

\begin{document}
\maketitle

\begin{abstract}
In this paper, we discuss the convergence rate of empirical processes of Gaussian processes for a large class of function families. Our main goal is to show the tail of the random quantity $\sup_{f \in F}| \widehat{P_n}(f) - \Eoo[\widehat{P_n}(f)] |$ can be dominated by polynomials. We put forward the properties of Hermite polynomials which play a crucial role in the proof of main theorems. At the end of the paper, we show the expectation of the random quantity converges to zero at the rate of $o(n^{-\frac{1}{2}+\varepsilon})$, which is proven to be $o(n^{-\frac{1}{3}})$ in \cite{paper1}.

\begin{flushleft}
\vspace{1mm}
\end{flushleft}

\end{abstract}

\section{Introduction}
\noindent In the field of probability theory, the tail estimate of a given stochastic process is one of the most basic and important problems. Specifically, if given a sequence of i.i.d random variables, for any bounded function family $F$ such that $0\in F$, we have$$\Eoo\sup_{f \in F}| \widehat{P_n}(f) - \Eoo[\widehat{P_n}(f)] |\leq C_2(F)n^{-\frac{1}{2}}+C_1(F)n^{-1},$$where $C_2$ and $C_1$ are two constants only relying on the selection of $F$; see \cite{tala}].\par
We define the size of a given empirical process as $\Eoo\sup_{f \in F}| \widehat{P_n}(f) - \Eoo[\widehat{P_n}(f)] |$. Since it is easy to show the upper bound of the convergence rate of $\Eoo\sup_{f \in F}| \widehat{P_n}(f) - \Eoo[\widehat{P_n}(f)] |$ cannot be better than $o(n^{-\frac{1}{2}})$, the result above tells us that under the condition of i.i.d., the structure of the size is extremely simple. To prove the theorem above, it is necessary to employ a famous result in tail estimate demonstrated as follows, which is the so-called Bernstein's inequality,$$\Poo(|\sum_i Y_i|\geq v)\leq 2\exp(-\min(\frac{v^2}{4\sum \Eoo Y_i^2},\frac{v}{2U}))$$ holds for any number $v\geq 0$ and any sequence of independent variables $Y_i (1\leq i\leq n)$.\par
Recalling that Bernstein's inequality is essentially necessary for the proof, one can observe intuitively that by modifying the independent condition of the given Gaussian process slightly, a similar inequality might still be provable. However, this problem is extremely difficult. A well-known consequence shows that an exponential bound still holds for all $L$-Lipschitz functions (see \cite{clas}). In 2017, a result established by Paouris and Valettas in \cite{conv} extended the previous result to all convex functions and improved the exponential bound by adopting a variance-sensitive form. However, these results are not even applicable to the most common case -- the indicators. Additionally, in this paper, we are going to prove the best convergence rate for indicators. Through some further assumptions, some papers show the Bernstein-like inequality holds for all measurable functions. For example, in \cite{mart}, the authors prove the inequality by employing the method of bounded martingale differences, and in \cite{markov}, a similar inequality for Markov chains is proven.\par
A good convergence rate still holds in some special cases without the establishment of Bernstein's inequality; see \cite{depe}, Chapter 1. Let $[0,1)$ be the unit interval equipped with Lebesgue measure and $\{x_k, k\geq 1\}$ be a sequence of random variables defined as $x_k := n_kx$, where $1 = n_1 \leq n_2 \leq n_3 \leq\cdots$, then exploiting the dynamical property of this sequence and some powerful tools from Fourier analysis, Roger Baker \cite{baker} proved $$D_n = o(n^{-\frac{1}{2}}(\log n)^{\frac{3}{2}+\varepsilon})$$holds for any $\varepsilon > 0.$\par
In our paper, we do not prove the Bernstein-like inequality but bring in a basic tool in the probability theory, which is called chaining argument or chaining method; see \cite{tala} or \cite{emtex}, to dominate the tail of empirical processes by polynomials. And this still implies the size of empirical processes converges at a rate close to the best one $O(n^{-\frac{1}{2}})$.

\section{Notations}
\noindent Throughout this paper, we adopt the notations from \cite{tala} and \cite{emtex}.
\begin{definition}[Empirical Process]
For a sequence of random variables $\{X_k\}$, the empirical process of it is defined as $$\widehat{P_n}(f) = \frac{1}{n}\sum_{k = 1}^{n}f(X_k),$$ where $f\in F$ and $F$ is a fixed set of measurable function. 
\end{definition}

\begin{definition}[Hermite Polynomials]
For $k \geq 0$, we define the $k^{th}$ Hermite polynomial $H_k$ by $$H_k = \frac{(-1)^{k}}{\sqrt{k!}}\exp(\frac{t^2}{2})\frac{d^k}{dt^k}\exp(-\frac{t^2}{2}).$$  
\end{definition}

\begin{definition}[Entropy Numbers]
Given a compact set $F$, the $n^{th}$ entropy number $e_n(F)$ is defined as follows, $$e_n(F) = \inf_{T_n\in F}\sup_{x\in F}d(x,T_n),$$ where $T_n$ represents any subset of $F$ with n elements.  
\end{definition}

\begin{definition}[Size]
We define the size of a given empirical process $\widehat{P_n}(f)$ as $\Eoo\sup_{f \in F}| \widehat{P_n}(f) - \Eoo(\widehat{P_n}(f)) |$. 
\end{definition}

\section{Properties of Hermite Polynomials}
\noindent We list three useful properties of Hermite polynomials without proof here. See the appendix for complete proof.
\begin{theorem}\label{pro1}
For real $\lambda$, $t$ and $k^{th}$ Hermite polynomials $H_k$, $$\exp(\lambda t - \frac{1}{2}\lambda^2) = \sum_{k=0}^{\infty}\frac{1}{\sqrt{k!}}H_k(t)\lambda^k.$$
\end{theorem}

\begin{theorem}\label{pro2}
Let $\gamma$ be the Gaussian measure with density function $p(t) = \frac{1}{\sqrt{2\pi}\sigma}\exp(-\frac{t^2}{2\sigma^2})$. Then $$\{ H_k(\frac{t}{\sigma}): k\geq 0\}$$ is an orthonormal basis for $L^2(\gamma).$
\end{theorem}

\begin{theorem}\label{pro3}
Suppose that a Gaussian vector $(U,V)$ satisfies
\begin{equation*}
(U,V)\sim N\left(0, \left( 
\begin{matrix}
\sigma_{11}&\sigma_{12}\\
\sigma_{12}&\sigma_{22}
\end{matrix}\right)\right). 
\end{equation*}

Then, for any positive integers $k,l\geq 0$, 
$$\Eoo(H_k(\frac{U}{\sqrt{\sigma_{11}}})H_l(\frac{V}{\sqrt{\sigma_{22}}})) = \delta_{k,l}(\frac{\sigma_{12}}{\sqrt{\sigma_{11}\sigma_{22}}})^k.$$

\end{theorem}

\section{Main Theorem}
\noindent Suppose that $\gamma$ is a Gaussian measure with a fixed variance $\sigma$ and $\{ X_n, n>0 \}$ is a centered Gaussian process such that each $X_n$ obeys $X_n\sim \mathcal{N}(0,\sigma)$ in the rest of the article.
\begin{lemma}\label{lemm}
Let $f, g$ be two arbitrary functions in $L^2(\gamma)$ with expectation 0. Then $$||\widehat{P_n}(f) - \widehat{P_n}(g)||_2 \leq \frac{\sqrt{\Delta_n}}{n}||f-g||_2,$$ where $\Delta_n = \sum_{i,j = 1}^n d_{ij}$ and $d_{ij} = \frac{\sigma_{ij}}{\sigma}.$
\end{lemma}

\begin{proof}
Combining the definition of $\widehat{P_n}(f)$ and Theorem \ref{pro2}, we obtain $$f(X_i) - g(X_i) = \sum_{k\geq 1}c_kH_k(\frac{X_i}{\sigma}),$$and
\begin{equation*}
\begin{aligned}
\Eoo[\widehat{P_n}(f) - \widehat{P_n}(g)]^2
&=\frac{1}{n^2}\Eoo[\sum_{k,l,i,j}c_kc_lH_k(\frac{X_i}{\sigma})H_l(\frac{X_j}{\sigma})] \\
&=\frac{1}{n^2}\sum_{k,i,j}c_k^2d_{ij}^k \\
&\leq \frac{\Delta_n}{n^2}\sum_kc_k^2 =\frac{\Delta_n}{n^2}||f-g||_2^2.
\end{aligned} 
\end{equation*}
\end{proof}

Now we are able to prove the following tail inequality. 

\begin{theorem}
Suppose that $F$ is a compact subset in $L^2(\gamma)$ with entropy numbers $\{e_n\}$ and $\{0\}=\pi_0\subset \pi_1\subset\cdots\subset\pi_n\subset\cdots$ is a sequence of nets in $F$ with $\card(\pi_n)=J_n$. Assume that $\lim_{n\to\infty}d(f,\pi_n) = 0$ holds for any $f\in F$. Then, we have $$\Poo(\sup_{f\in F}|\widehat{P_n}(f) - E[\widehat{P_n}(f)]|\geq \lambda)\leq C_1^2(e,q)C_2(\pi,q)\frac{\Delta_n}{\lambda^2n^2},$$ where $C_1, C_2$ are two constants depending only on nets $\pi$, entropy numbers $e_n$ and a non-decreasing sequence of positive numbers $q_n$ taken arbitrarily.
\end{theorem}

\begin{proof}
Without loss of generality, we suppose $\Eoo f = 0$ for any $f\in F$. Let $\pi_n(f)$ denote the closest element to $f$ in $\pi_n$. Since $\lim_{n\to\infty}d(f,\pi_n) = 0$, it is obvious that $\forall f\in F$, we have $$f = \sum_{k\geq 0}(\pi_{k+1}(f)-\pi_k(f)).$$ Then, $$|f|\leq\sum_{k\geq 0}|\pi_{k+1}(f)-\pi_k(f)|,$$ and
 \begin{equation}\label{summ}
|\widehat{P_n}(f)|\leq\sum_{k\geq 0}|\widehat{P_n}(\pi_{k+1}(f))-\widehat{P_n}(\pi_{k}(f))|.
\end{equation}
Lemma \ref{lemm} implies $$\Poo(|\widehat{P_n}(\pi_{k+1}(f))-\widehat{P_n}(\pi_{k}(f))|>\lambda)\leq\frac{\Delta_n||\pi_{k+1}(f)-\pi_{k}(f)||_2^2}{n^2\lambda^2}.$$ Applying variable substitution, we obtain $$\Poo(|\widehat{P_n}(\pi_{k+1}(f))-\widehat{P_n}(\pi_{k}(f))|>\lambda q_k||\pi_{k+1}(f)-\pi_{k}(f)||_2\frac{\sqrt{\Delta_n}}{n})\leq\frac{1}{\lambda^2q_k^2},$$ where $q_k$ is a non-decreasing sequence of positive numbers taken arbitrarily. Then, we denote by $\Omega_\lambda$ the set such that $$\forall k\geq 0, \forall f\in F, |\widehat{P_n}(\pi_{k+1}(f))-\widehat{P_n}(\pi_{k}(f))|\leq\lambda q_k||\pi_{k+1}(f)-\pi_{k}(f)||_2\frac{\sqrt{\Delta_n}}{n} $$ and observe $$P(\Omega_\lambda^c)\leq\sum_{k\geq 0}\frac{N_k}{\lambda^2q_k^2},$$where $N_k$ is the number of different pairs $(\pi_{k+1}(f),\pi_k(f))$ while taking different $f$.\par
Here we define $C_1,C_2$ by $$C_1(e,q):=2\sum_{k\geq 0}e_{j_k}(F)q_k$$ and $$C_2(\pi,q):=\sum_{k\geq 0}\frac{N_k}{q_k^2}.$$When $\Omega_\lambda$ occurs, formula (\ref{summ}) implies $$\forall f\in F, |\widehat{P_n}(f)|\leq\sum_{k\geq 0}\lambda q_k||\pi_{k+1}(f)-\pi_{k}(f)||_2\frac{\sqrt{\Delta_n}}{n}.$$ Thus,  $$\sup_{f\in F}|\widehat{P_n}(f)|\leq\sup_{f\in F}\sum_{k\geq 0}\lambda q_k||\pi_{k+1}(f)-\pi_{k}(f)||_2\frac{\sqrt{\Delta_n}}{n},$$ which implies 
\begin{equation*}
\begin{aligned}
\Poo(\sup_{f\in F}|\widehat{P_n}(f)|\geq\sum_{k\geq 0}\lambda q_k\sup_{f\in F}||\pi_{k+1}(f)-\pi_{k}(f)||_2\frac{\sqrt{\Delta_n}}{n})&\leq \Poo(\Omega_\lambda^c)\\
&\leq \frac{C_2(\pi,q)}{\lambda^2}.
\end{aligned} 
\end{equation*}
Further, since $$\sum_{k\geq 0}\lambda q_k\sup_{f\in F}||\pi_{k+1}(f)-\pi_{k}(f)||_2\frac{\sqrt{\Delta_n}}{n}\leq 2\sum_{k\geq 0}\lambda q_k\sup_{f\in F}||f-\pi_{k}(f)||_2\frac{\sqrt{\Delta_n}}{n},$$ by choosing $\pi_n$ properly, we obtain $$\Poo(\sup_{f\in F}|\widehat{P_n}(f)|\geq\lambda C_1(e,q)\frac{\sqrt{\Delta_n}}{n})\leq\frac{C_2(\pi,q)}{\lambda^2}.$$ This implies the theorem.
\end{proof}

\begin{corol}
The size of the empirical process above satisfies $$\Eoo(\sup_{f\in F}|\widehat{P_n}(f)-\Eoo[\widehat{P_n}(f)]|)\leq 2C_1\sqrt{C_2}\frac{\sqrt{\Delta_n}}{n}.$$
\end{corol}

\begin{proof}
Suppose $\forall f\in F, \Eoo (f) = 0$. Since $\Eoo X = \int_{0}^{\infty}\Poo(X>u)du$ holds for any variable $X\geq 0$, applying main theorem, we obtain
\begin{equation*}
\begin{aligned}
\Eoo(\sup_{f\in F}|\widehat{P_n}(f)|)&\leq \int_{0}^{u}\Poo(\sup_{f\in F}|\widehat{P_n}(f)|>x)dx+\int_{u}^{\infty}C_1^2C_2\frac{\Delta_n}{x^2n^2}dx\\
&\leq u+C_1^2C_2\frac{\Delta_n}{un^2}
\end{aligned} 
\end{equation*}
By taking $u$ properly, we get $$\Eoo(\sup_{f\in F}|\widehat{P_n}(f)|)\leq 2C_1\sqrt{C_2}\frac{\sqrt{\Delta_n}}{n}.$$
\end{proof}

Note that if $C_1, C_2$ are two finite numbers, we can dominate the size of $\widehat{P_n}(f)$ by $O(\frac{\sqrt{\Delta_n}}{n})$. However, even in the case of $F = \{1_{\{t\leq x\}}-\Phi(x):x\in\mathbb{R}\}$, we cannot take a sequence of $q_k$ directly such that $C_1, C_2$ are convergent.

\section{Upper Bound for Empirical Processes under General Dependence}

\noindent In this section, we are going to prove that for the set of indicators, the size of the empirical process can still be dominated by $O(\frac{\sqrt{\Delta_n}}{n})$, and we suppose $F$ is the set of indicators in the rest of article.

\begin{lemma}\label{le222}
For any sufficiently small $\delta > 0$, there exists a finite subset $F_\delta\in F$ such that $$\Eoo(\sup_{f\in F_\delta}|\widehat{P_n}(f)|)\leq C(\log\frac{1}{\delta})^{\frac{3}{2}}\frac{\sqrt{\Delta_n}}{n},$$ and $$\sup_{f\in F}d(f,F_\delta)\leq 2\delta.$$
\end{lemma}

\begin{proof}
Choose an integer $m$ such that $\frac{1}{2^{m+1}}<\delta\leq\frac{1}{2^m}$ and $$F_\delta = \{1_{\{x\leq\alpha_i\}}-\frac{i}{2^m}:0\leq i\leq 2^m\}$$ where $\alpha_i$ is a real number such that $\Phi(\alpha_i) = \frac{i}{2^m}.$ Choose a sequence of nets by $$\{0\}=\pi_0\subset \pi_1\subset\cdots\subset\pi_m = \pi_{m+1} =\cdots$$ where $$\pi_k(k\leq m) = \{1_{\{x\leq\alpha_{i2^{m-k}}\}}-\frac{i}{2^k}:0\leq i\leq 2^k\}.$$
It is obvious that $J_k(0<k\leq m) = 1+2^k$ and $J_k(k>m) = J_m$. As a result, we have $$C_2 = C(\sum_{k = 0}^{m}\frac{2^k}{q_k^2}+\sum_{k>m}\frac{2^m}{q_k^2})$$ since there is a constant such that $N_k\leq C2^k$ for $k\leq m$. Similarly, we have $$C_1 = \sum_{k=0}^{m}\frac{q_k}{\sqrt{2^k}}.$$ By main theorem, since we can assign an arbitrarily huge value to $q_k$ for $k>m$, we obtain $$\Eoo\sup_{f\in F_\delta}|\widehat{P_n}(f)|\leq C\sqrt{(\sum_{k = 0}^{m}\frac{2^k}{q_k^2})}(\sum_{k=0}^{m}\frac{q_k}{\sqrt{2^k}})\frac{\sqrt{\Delta_n}}{n}.$$ By choosing $q_k$ properly, the inequality above implies $$\Eoo\sup_{f\in F_\delta}|\widehat{P_n}(f)|\leq Cm^{\frac{3}{2}}\frac{\sqrt{\Delta_n}}{n}.$$ As $\frac{1}{2^m}\geq \delta$, we have $$m\leq \log(\frac{1}{\delta}).$$
\end{proof}

\begin{lemma}\label{kl}
For any sufficiently small $\delta > 0$, the following inequality holds, $$\Eoo(\sup_{f\in F}|\widehat{P_n}(f)|)\leq C((\log\frac{1}{\delta})^{\frac{3}{2}}\frac{\sqrt{\Delta_n}}{n}+\delta).$$
\end{lemma}

\begin{proof}
For convenience, we represent $\widehat{P_n}(f)$ by $\widehat{P_n}(\alpha)$ when $f=1_{\{t\leq\alpha\}}-\Phi(\alpha)$. Suppose $$F_\delta = \{1_{\{x\leq\alpha_i\}}-\Phi(\alpha_i):0 = \alpha_0 < \alpha_1< \alpha_2<\cdots \}$$ which is the subset of $F$ defined in Lemma \ref{le222}. Define a new process by $$\widehat{Q_n}(\alpha) = \widehat{P_n}(1_{\{t\leq\alpha_i\}}), \alpha_{i-1}\leq\alpha < \alpha_i.$$ Since $$\widehat{P_n}(\alpha) = \widehat{P_n}(1_{\{t\leq\alpha\}})-\Phi(\alpha),$$ by the definition of $\alpha_i$, we have 
\begin{equation*}
\begin{aligned}
\Eoo(\sup_{\alpha}|\widehat{Q_n}(\alpha)-\Eoo\widehat{P_n}(1_{t\leq\alpha})|)&\leq \Eoo(\sup_{\alpha}|\widehat{Q_n}(\alpha)-\Eoo\widehat{P_n}(1_{t\leq\alpha_i})|) + 2\delta\\
&\leq \Eoo(\sup_{\alpha_i}|\widehat{P_n}(\alpha_i)|) + 2\delta\\
&\leq C((\log\frac{1}{\delta})^{\frac{3}{2}}\frac{\sqrt{\Delta_n}}{n}+\delta).
\end{aligned} 
\end{equation*}

Further, we have $$\Eoo(\sup_{\alpha}|\widehat{Q_n}(\alpha)-\Eoo\widehat{P_n}(1_{t\leq\alpha})|^+)\leq C((\log\frac{1}{\delta})^{\frac{3}{2}}\frac{\sqrt{\Delta_n}}{n}+\delta)$$ and $$\Eoo(\sup_{\alpha}|\widehat{Q_n}(\alpha)-\Eoo\widehat{P_n}(1_{t\leq\alpha})|^-)\leq C((\log\frac{1}{\delta})^{\frac{3}{2}}\frac{\sqrt{\Delta_n}}{n}+\delta)$$ A useful observation is that $$\widehat{Q_n}(\alpha-2\delta)\leq\widehat{P_n}(1_{t\leq\alpha})\leq\widehat{Q_n}(\alpha + 2\delta)$$ which implies 
\begin{equation*}
\begin{aligned}
\Eoo(\sup_{\alpha}|\widehat{P_n}(1_{t\leq\alpha})-\Eoo\widehat{P_n}(1_{t\leq\alpha+2\delta})|^+)&\leq \Eoo(\sup_{\alpha}|\widehat{Q_n}(\alpha+2\delta)-\Eoo\widehat{P_n}(1_{t\leq\alpha+2\delta})|^+)\\
&\leq C((\log\frac{1}{\delta})^{\frac{3}{2}}\frac{\sqrt{\Delta_n}}{n}+\delta)
\end{aligned} 
\end{equation*}
and
\begin{equation*}
\begin{aligned}
\Eoo(\sup_{\alpha}|\widehat{P_n}(1_{t\leq\alpha})-\Eoo\widehat{P_n}(1_{t\leq\alpha-2\delta})|^-)&\leq \Eoo(\sup_{\alpha}|\widehat{Q_n}(\alpha-2\delta)-\Eoo\widehat{P_n}(1_{t\leq\alpha-2\delta})|^-)\\
&\leq C((\log\frac{1}{\delta})^{\frac{3}{2}}\frac{\sqrt{\Delta_n}}{n}+\delta)
\end{aligned} 
\end{equation*}
By the definition of $\widehat{P_n}$, we have $$\Eoo\widehat{P_n}(1_{t\leq\alpha+2\delta})\leq \Eoo\widehat{P_n}(1_{t\leq\alpha})+2\delta$$ and $$\Eoo\widehat{P_n}(1_{t\leq\alpha-2\delta})\geq \Eoo\widehat{P_n}(1_{t\leq\alpha})-2\delta.$$ So, 
\begin{equation*}
\begin{aligned}
\Eoo(\sup_{\alpha}|\widehat{P_n}(1_{t\leq\alpha})-\Eoo\widehat{P_n}(1_{t\leq\alpha})|^+)&\leq \Eoo\sup_{\alpha}(|\widehat{P_n}(1_{t\leq\alpha})-\Eoo\widehat{P_n}(1_{t\leq\alpha+2\delta})|+2\delta)^+\\
&\leq \Eoo\sup_{\alpha}(|\widehat{P_n}(1_{t\leq\alpha})-\Eoo\widehat{P_n}(1_{t\leq\alpha+2\delta})|^++2\delta\\
&\leq C((\log\frac{1}{\delta})^{\frac{3}{2}}\frac{\sqrt{\Delta_n}}{n}+\delta)
\end{aligned} 
\end{equation*}
and
\begin{equation*}
\begin{aligned}
\Eoo(\sup_{\alpha}|\widehat{P_n}(1_{t\leq\alpha})-\Eoo\widehat{P_n}(1_{t\leq\alpha})|^-)&\leq \Eoo\sup_{\alpha}(|\widehat{P_n}(1_{t\leq\alpha})-\Eoo\widehat{P_n}(1_{t\leq\alpha-2\delta})|+2\delta)^-\\
&\leq \Eoo\sup_{\alpha}(|\widehat{P_n}(1_{t\leq\alpha})-\Eoo\widehat{P_n}(1_{t\leq\alpha-2\delta})|^-+2\delta\\
&\leq C((\log\frac{1}{\delta})^{\frac{3}{2}}\frac{\sqrt{\Delta_n}}{n}+\delta).
\end{aligned} 
\end{equation*}

Combining the two inequalities above, we obtain 
\begin{equation*}
\begin{aligned}
\Eoo(\sup_{\alpha}|\widehat{P_n}(1_{t\leq\alpha})-\Eoo\widehat{P_n}(1_{t\leq\alpha})|)\leq &\Eoo(\sup_{\alpha}|\widehat{P_n}(1_{t\leq\alpha}) -\Eoo\widehat{P_n}(1_{t\leq\alpha})|^+) \\
&\quad +\Eoo(\sup_{\alpha}|\widehat{P_n}(1_{t\leq\alpha})-\Eoo\widehat{P_n}(1_{t\leq\alpha})|^-)\\
 \leq &C((\log\frac{1}{\delta})^{\frac{3}{2}}\frac{\sqrt{\Delta_n}}{n}+\delta).
\end{aligned} 
\end{equation*}
\end{proof}

\begin{theorem}
For any $0<\varepsilon<1$, the size of the empirical process satisfies $$\Eoo(\sup_{f\in F}|\widehat{P_n}(f)|)\sim o((\frac{\sqrt{\Delta_n}}{n})^{\varepsilon}).$$
\end{theorem}

\begin{proof}
By Lemma \ref{kl}, we have 
\begin{equation}\label{ff}
\forall\delta > 0, \Eoo(\sup_{f\in F}|\widehat{P_n}(f)|)\leq C((\log\frac{1}{\delta})^{\frac{3}{2}}\frac{\sqrt{\Delta_n}}{n}+\delta).
\end{equation}
By transforming the right-hand side of (\ref{ff}) to a better form, we need to estimate the order of $$x^{\frac{3}{2}}\frac{\sqrt{\Delta_n}}{n}+e^{-x}.$$ Since there exists a constant $c(p)$ for every $p\in\mathbb{N}$ such that $e^{-x}\leq c(p)x^{-p}$, we obtain $$x^{\frac{3}{2}}\frac{\sqrt{\Delta_n}}{n}+e^{-x}\leq C(\frac{\sqrt{\Delta_n}}{n})^{\frac{2p}{2p+3}}$$ by assigning a proper value to $x$. The proof is completed as $p$ can be taken arbitrarily.
\end{proof}

\appendix
\section{Appendix}

\begin{flushleft} {\it Proof of Theorem \ref{pro1}.}\ \ For a real number $u$, let $g(u) = \exp(- \frac{1}{2}u^2)$. For $t\in \mathbb{R}$,
\begin{equation*}
\begin{aligned}
g(u)&=\sum_{k=0}^{\infty}\frac{g^{(k)}(t)}{k!}(u-t)^k =\sum_{k=0}^{\infty}\frac{\sqrt{k!}}{(-1)^k}\exp(-\frac{t^2}{2})H_k(t)\frac{1}{k!}(u-t)^k\\
&=\exp(-\frac{t^2}{2})\sum_{k=0}^{\infty}\frac{(-1)^k}{\sqrt{k!}}H_k(t)(u-t)^k
\end{aligned} 
\end{equation*}
Hence, for real $\lambda$ and $t$,
\begin{equation*}
\begin{aligned}
\exp(\lambda t - \frac{1}{2}\lambda^2) &= \exp(\frac{1}{2}t^2 - \frac{1}{2}(\lambda^2-t^2)) = \exp(\frac{1}{2}t^2)g(t-\lambda)\\
&= \exp(\frac{1}{2}t^2)\exp(-\frac{t^2}{2})\sum_{k=0}^{\infty}\frac{(-1)^k}{\sqrt{k!}}H_k(t)(-\lambda)^k
= \sum_{k=0}^{\infty}\frac{1}{\sqrt{k!}}H_k(t)\lambda^k. \qquad \Box
\end{aligned} 
\end{equation*}

\end{flushleft}

\begin{flushleft} {\it Proof of Theorem \ref{pro2}}\ \ 
For any $\lambda, \mu\in \mathbb{R}$, on the one hand,
\begin{equation*}
\begin{aligned}
\int_{\mathbb{R}}\exp(\lambda t - \frac{\lambda^2\sigma^2}{2})\exp(\mu t - \frac{\mu^2\sigma^2}{2})d\gamma 
=&e^{\sigma^2\lambda\mu}\int_{\mathbb{R}}\frac{1}{\sqrt{2\pi} \sigma}\exp(-\frac{(t-\sigma^2(\lambda+\mu))^2}{2\sigma^2})dt.\\
=&e^{\sigma^2\lambda\mu}.
\end{aligned} 
\end{equation*}
On the other hand, substituting $t$ and $\lambda$ by $\frac{t}{\sigma}$ and $\sigma\lambda$ in Theorem \ref{pro1}, we have

\begin{equation*}
\begin{aligned}
\int_{\mathbb{R}}\exp(\lambda t - \frac{\lambda^2\sigma^2}{2})\exp(\mu t - \frac{\mu^2\sigma^2}{2})d\gamma
=&\int_{\mathbb{R}}(\sum_{k=0}^{\infty}\frac{1}{\sqrt{k!}}H_k(\frac{t}{\sigma})\lambda^k\sigma^k)(\sum_{l=0}^{\infty}\frac{1}{\sqrt{l!}}H_l(\frac{t}{\sigma})\mu^l\sigma^l)d\gamma\\
=&\int_{\mathbb{R}}(\sum_{k,l\geq 0}\frac{1}{\sqrt{k!l!}}\sigma^{k+l}\lambda^k\mu^lH_k(\frac{t}{\sigma})H_l(\frac{t}{\sigma}))d\gamma\\
=&\sum_{k,l\geq 0}\frac{1}{\sqrt{k!l!}}\sigma^{k+l}\lambda^k\mu^l\langle H_k(\frac{t}{\sigma}),H_l(\frac{t}{\sigma})\rangle_{L^2(\gamma)}.
\end{aligned} 
\end{equation*}
Therefore, $$e^{\sigma^2\lambda\mu} = \sum_{k,l\geq 0}\frac{1}{\sqrt{k!l!}}\sigma^{k+l}\lambda^k\mu^l\langle H_k(\frac{t}{\sigma}),H_l(\frac{t}{\sigma})\rangle_{L^2(\gamma)}.$$ This completes the proof.\hspace{\fill} $\square$
\end{flushleft}

\begin{flushleft} {\it Proof of Theorem \ref{pro3}}\ \ 
Simply adopt the same routine as the proof of \ref{pro2}.\hspace{\fill} $\square$
\end{flushleft}

\end{document}